\documentclass[reqno, 11pt]{amsart}
\usepackage{amssymb}
\usepackage[numbers]{natbib}
\pdfoutput=1
\usepackage[nesting]{hyperref}

\usepackage[pdftex]{graphicx}
\usepackage{listings}
\usepackage{multirow}
\usepackage{placeins}
\usepackage{color}
\usepackage{subfigure}
\usepackage{lscape}

\newcommand{\BlackBoxes}{\global\overfullrule5pt}

\BlackBoxes

\newcommand{\R}{\mathbb{R}}
\newcommand{\N}{\mathbb{N}}
\newcommand{\Z}{\mathbb{Z}}
\newcommand{\E}{\mathbb{E}}

\newcommand{\PP}{\mathbb{P}}
\newcommand{\Eop}{\operatorname{\mathbb{\E}}}

\newcommand{\Pop}{\operatorname{\mathbb{\PP}}}

\renewcommand{\P}{\mathbb{P}}

\newtheorem{theorem}{Theorem}
\newtheorem{corollary}[theorem]{Corollary}
\newtheorem{lemma}[theorem]{Lemma}

\theoremstyle{definition}

\newtheorem{remark}[theorem]{Remark}
\newtheorem{definition}[theorem]{Definition}
\numberwithin{equation}{section} \numberwithin{theorem}{section}

\def\0{\kern0pt\-\nobreak\hskip0pt\relax}

\makeatletter
\AtBeginDocument{%
 \def\@serieslogo{%
 \vbox to\headheight{%
 \parindent\z@ \fontsize{6}{7\p@}\selectfont
\today $\;$ 
 \vss}}}

\def\makeoverbar#1#2#3#4#5#6#7{%
 \setbox0=\hbox{$\m@th#2\mkern#5mu{{}#3{}}\mkern#6mu$}%
 \setbox1=\null \dimen@=#4\fontdimen8#13 \dimen@=3.5\dimen@
 \advance\dimen@ by \ht0 \dimen@=-#7\dimen@ \advance\dimen@ by \wd0
 \ht1=\ht0 \dp1=\dp0 \wd1=\dimen@
 \dimen@=\fontdimen8#13 \fontdimen8#13=#4\fontdimen8#13
 \rlap{\hbox to \wd0{$\m@th\hss#2{\overline{\box1}}\mkern#5mu$}}
 \fontdimen8#13=\dimen@}

\def\mylabel#1#2{{\def\@currentlabel{#2}\label{#1}}}

\makeatother

\begin{document}


\makeatletter \providecommand\@dotsep{5} \makeatother

\title[Risk-Sensitive Dividend Problems]{Risk-Sensitive Dividend Problems}

\author[N. \smash{B\"auerle}]{Nicole B\"auerle${}^*$}
\address[N. B\"auerle]{Department of Mathematics,
Karlsruhe Institute of Technology, D-76128 Karlsruhe, Germany}

\email{\href{mailto:nicole.baeuerle@kit.edu}
{nicole.baeuerle@kit.edu}}

\author[A. \smash{Ja\'{s}kiewicz}]{Anna Ja\'{s}kiewicz $^\ddag$}
\address[A. Ja\'{s}kiewicz]{Institute of Mathematics and Computer Science,
Wroc{\l}aw University of Technology, PL-50-370 Wroc{\l}aw, Poland}

\email{\href{mailto:anna.jaskiewicz@pwr.edu.pl} {anna.jaskiewicz@pwr.edu.pl}}

\thanks{${}^\ddag$ The second author is grateful to the Alexander von Humboldt Foundation for supporting her
research stay at the Institute of Stochastic at the KIT. This work was also partially supported
by the National Science Centre [Grant DEC-2011/03/B/ST1/00325].}

\begin{abstract}
We consider a discrete time version of the popular optimal dividend payout problem in risk theory.
The novel aspect of our approach is that we allow for a risk averse insurer, i.e., instead of maximising the expected discounted
dividends until ruin we maximise the expected {\em utility} of  discounted dividends until ruin.
This task has been proposed as an open problem in \cite{gershiu}.
The model in a continuous-time Brownian motion setting with the exponential utility
function  has been analysed in \cite{schach}. Nevertheless, a complete solution has not been provided.
In this work, instead we  solve the problem in discrete time setup for the exponential and the power utility functions and give the structure
of optimal history-dependent dividend policies.
We make use of certain ideas studied earlier in \cite{br_mor}, where Markov decision processes with general utility functions
were treated. Our analysis, however,  include new aspects, since the reward functions in this case are not bounded.

\end{abstract}
\maketitle

\vspace{0.5cm}
\begin{minipage}{14cm}
{\small
\begin{description}
\item[\rm \textsc{ Key words} ]
{\small Markov decision process, Dividend payout, Risk aversion, History-dependent policy, Fixed point problem}
\item[\rm \textsc{AMS subject classifications} ]
{\small  90C40, 91B30 }
\end{description}
}
\end{minipage}

\section{Introduction}\label{sec:intro}
The dividend payout problem in risk theory has been introduced by de Finetti \cite{deFinetti} and has since
then been investigated under various extensions
during the decades up to now; see, for instance, \cite{turkey, schach}.
The task is to find in a given model for the free surplus process of an insurance company,
a dividend payout strategy that maximises  the expected discounted dividends until ruin.
Typical models for the surplus process are compound Poisson processes, diffusion processes, general renewal processes
or discrete time processes. The reader is referred  to \cite{albrthonh} and \cite{avanzi}, where  an excellent overview of recent results
is provided.

In \cite{gershiu} the authors propose the problem of maximising the expected {\em utility} of  discounted dividends until ruin instead of
maximising the expected discounted dividends until ruin.
This means that  an insurance company is equipped with some utility function that helps it to measure the accumulated dividends paid to
the shareholders. If this utility is increasing and concave, the company is risk averse (see Remark 2.2).
To the best of our knowledge, there is only one work \cite{schach},
in which this idea was taken up. More precisely, the  authors in \cite{schach} consider a linear Brownian motion model
for the free surplus process and apply the exponential utility function to evaluate the  discounted dividends until ruin. It turns out that
the mathematics involved in the  analysis of this problem is quite different from the one used in the risk neutral case and
only partial results could be obtained. In contrast to the same problem with a risk neutral insurance company,
where the optimal dividend payout strategy is a barrier strategy (see e.g., \cite{at}), the authors in \cite{schach}
are not able to identify the structure of the optimal dividend policy rigorously. They show imposing
some further assumptions  that  there is  a time dependent optimal barrier.

We study the same problem but with a discrete time surplus process.
The risk neutral problem within such a framework  can be found in  Section 9.2 in  \cite{br} or in Section 1.2 in  \cite{s}.
By making use of the dynamic programming approach the authors in  \cite{br}  and  \cite{s} prove that the optimal dividend payout
policy is a stationary band-strategy. In \cite{abt}, on the other hand, the authors  consider a discrete time model
that is formulated with the aid of a general L\'evy surplus process but the dividend payouts are allowed
only at random discrete time points.  This  version  can again be solved  by the dynamic programming arguments.
However, the problem with a general utility function is more demanding.
Like in the continuous time setting \cite{schach}, it requires a sophisticated analysis.
It is worth mentioning that   Markov decision processes with
general utility functions
have been already studied in \cite{br_mor,kkm}.  Moreover, there are also some papers, where the specific utility functions
are considered. For example, Jaquette  \cite{j1,j2} and Chung and Sobel \cite{cs} are among the first who examined
discounted payoffs in  Markov decision processes with the decision maker that is equipped with a constant risk aversion, i.e.,
grades her random payoffs with the help of the exponential utility function.
 The common feature of all the aforementioned papers is the fact that they deal with bounded rewards or costs.
 Therefore, their results cannot be directly applied to our case, where the payoffs are unbounded.
 We make use of the special structure of the underlying problem and show that the optimal dividend payout policy
is a time dependent band-strategy. The value function itself can be characterised as a solution
to a certain optimality equation.
Furthermore, we also study  the dividend payout model with the power utility function.
As noted in \cite{br_mor}, the original Markov decision process  can then be viewed as a Markov decision process
defined on the extended state space.
 We employ these techniques to solve our model, but only in the first step, where we use an approximation of the value
 function in the infinite time horizon by value functions in the finite time horizons.
 In contrast to the exponential utility case, we can only partly identify the structure of the optimal dividend payout  policy.
 However, we are able to show that there is a barrier such that when the surplus is above the barrier,
 it is always optimal to pay down to a state below the barrier.
 The value function is again  characterised as a solution to some optimality equation.
Summing up, the optimal dividend payout problem with the exponential utility function
can be solved completely in the discrete time case, in contrast to the  continuous-time problem in  \cite{schach},
whilst for the case with the power
utility function we are at least able to identify the important global structure of the optimal policy.

The paper is organised as follows. In the next section we introduce the model together with mild assumptions
and general history-dependent policies. Section 3 is devoted to a study of the exponential utility case.
We show first that the value function  $J$  for discounted payoffs   satisfies an optimality
equation and give a lower and an upper bound for $J$. Then, we identify properties of the minimiser of the right-hand side of
the optimality equation.
This enables us to show that the minimiser indeed defines an optimal policy, which is
a non-stationary band-policy. The non-stationarity is based only  on the time-dependence.
 The power utility case  is treated in Section 4.
We pursue here a little different approach, but it  also leads to an optimality equation.
The policies obtained in this setting are really
history-dependent. Nonetheless, we are still able to show that the optimal policy is of a barrier-type.
In Section 5 we provide the policy improvement algorithm for the model with the exponential utility.
Finally, Section 6 is devoted to concluding remarks and open issues.

\section{The Model}\label{sec:mod}
We  consider the financial situation of an insurance company at discrete times,
say $n \in\N_0:= 0, 1, 2,\ldots.$	
Assume there is an initial surplus $x_0 = x\in X:=\Z$ and $x_0\ge 0.$
The surplus $x_{n+1}$ at time $n+1$ evolves according to the following equation
 \begin{equation}
\label{maineq}
x_{n+1} = x_n-a_n+	
Z_{n+1},\; \mbox{if }x_n\ge 0\quad\mbox{and}\quad
x_{n+1} = x_n,\; \mbox{if }x_n< 0.
\end{equation}
Here $a_n\in A(x_n):=\{0,\ldots,x_n\}$  denotes the dividends paid to the shareholders at time $n,$ and $Z_{n+1}$ represents the income
(possibly negative) of the company during the time interval from $n$ to $n+1.$
More precisely, $Z_{n+1}$ is the
difference between premium and claim sizes in the $(n+1)$-st time interval.
Further, we assume that
$Z_1, Z_2,\ldots$ are independent and identically distributed integer-valued random variables with distribution
$(q_k)_{k\in \Z}$, i.e., $\Pop(Z_n=k)=q_k,$ $k\in \Z.$
A  dividend payout problem in the risk theory can be viewed as a Markov decision process with the state space  $X,$
the set of actions $A(x)$ available in state $x$ (for completeness, we put $A(x)=\{0\}$ for $x<0$)
and the transition probability $q(\cdot|x,a)$ of
the next state, when $x$ is the current state and $a$ is the amount of dividend paid to the shareholders.
Note that the dynamics of equation (\ref{maineq}) implies that
 $q(y|x,a)=q_{y-x+a}$ for $x\ge 0$ and $q(x|x,a)=1$ if $x<0.$
 For the set of admissible pairs $D:=\{(x,a): \;x\in X,\;a\in A(x)\}$ we define the function
 $r: D\mapsto \R$ as  $r(x,a)=a$ for $x\in X.$
	
The feasible history spaces are defined as follows $\Omega_0=X,$ $\Omega_k=D^k\times X$ and
$\Omega_\infty=D^\infty.$ A {\it policy} $\pi = (\pi_{k})_{k\in \N_0}$
is a sequence of  mappings from $\Omega_k$ to $A$ such that
$\pi_k(\omega_k)\in A(x_k),$ where $\omega_k=(x_0,a_0,\ldots,x_k)\in \Omega_k.$
Let $\Gamma$ be the class of all  functions $g:X\mapsto A$ such that $g(x)\in A(x).$
A Markov policy is $\pi=(g_k)_{k\in\N_0}$ where each $g_k\in \Gamma.$
By $\Pi$ and $\Pi^M$  we denote the set of all history-dependent and Markov policies, respectively.
By the Ionescu-Tulcea theorem \cite{n},
for each policy $\pi$ and each initial state $x_0=x,$ a probability measure $\P_x^\pi$
and a stochastic process $(x_k,a_k)_{k\in\N_0}$ are defined on $\Omega_\infty$
in a canonical way, where $x_k$ and $a_k$ describe the state and the decision at stage $k,$
respectively.
By $\Eop_x^\pi$ we denote the expectation operator with respect to the probability measure $\P_x^\pi.$

Ruin occurs as soon as the surplus gets negative.
The epoch $\tau$ of ruin is defined as the smallest integer $n$
such that $x_n <0.$
The question arises as to how the risk-sensitive insurance company, equipped with some
utility function will choose its dividend strategy. More precisely, we shall consider
the following optimisation problem
$$\sup_{\pi\in\Pi} \Eop_x^\pi U_\gamma\left(\sum_{k=0}^\infty \beta^k r(x_k,a_k)\right) =
\sup_{\pi\in\Pi} \Eop_x^\pi U_\gamma\left(\sum_{k=0}^{\tau-1}\beta^k a_k\right),\quad x\ge 0,$$
where $\beta\in (0,1)$ is a discount factor and  either
\begin{itemize}
	\item[(1)] $U_\gamma$ is the exponential utility function,
	i.e.,  $U_\gamma(x)=\frac 1{\gamma} e^{\gamma x}$ with $\gamma<0,$ or
	\item[(2)]  $U_\gamma$ is the power utility function, i.e.,  $U_\gamma(x)=x^{\gamma}$ with $\gamma\in (0,1).$
\end{itemize}

Let $Z$ be a random variable with the same distribution as $Z_1.$
Throughout the paper the following assumptions will be supposed to hold true.

\begin{itemize}
\item[(A1)]  $\Eop Z^+<+\infty,$ where $Z^+=\max\{Z,0\};$
\item[(A2)] $\Pop(Z<0)>0.$
\end{itemize}
Assumption (A2) allows to avoid a trivial case, when the ruin will never occur
under any policy $\pi\in \Pi.$

\begin{remark}
In our study, we assume that the random variables $\{Z_n\}$ only take integer values and the initial capital is
also integer.  From the proof of Lemma 1.9 in \cite{s}, it follows that in our problem
we can restrict without  loss of generality to the integer dividend payments.  \end{remark}

\begin{remark}
If the function $U_\gamma$ is strictly concave and increasing as in our case, then the quantity $U_\gamma^{-1} \big( \Eop[U_\gamma(X)]\big)$
is called a {\em certainty equivalent} of the random variable $X.$ From the optimisation's point of view  it does not matter which value
 $U_\gamma^{-1} \big( \Eop[U_\gamma(X)]\big)$  or $ \Eop[U_\gamma(X)]$ we study, because
the inverse function $U_\gamma^{-1}$ is monotonic.  However,  the certainty equivalent has an important meaning. If we apply the Taylor
expansion, then the certainty equivalent can be written as follows
$$ U^{-1}_\gamma\Big(\Eop\big[U_\gamma(X)\big]\Big)\approx \Eop X - \frac12 l(\Eop X) Var[X],$$
where $$ l(y) = -\frac{U''_\gamma(y)}{U'_\gamma(y)}$$ is called the {\em Arrow-Pratt} function of absolute risk aversion. Hence,
 the second term accounts for the variability of $X$ (for a discussion see \cite{bpli}). If $U_\gamma$ is concave like in our case,
 then $l(\cdot)\ge 0$ which means that the variance is subtracted. This fact implies that the decision maker is risk averse.
\end{remark}

\section{The Exponential Utility Function}
In this section we assume that the insurer is risk averse and grades her random payoffs
by taking the expectations
of the exponential utility function of these random rewards. More precisely,
we assume that the decision maker is
equipped with the constant risk coefficient $\gamma<0.$
The objective of the risk averse insurer is to maximise
the expected discounted payoff function:
$$ \widetilde{J}_\pi(x)= \Eop_x^\pi U_\gamma\left(\sum_{k=0}^{\infty}\beta^k r(x_k,a_k)\right)$$
and to find a policy $\pi^*\in\Pi$ (if exists) such that
\begin{equation}
\label{optpp}
\widetilde{J}(x):=\widetilde{J}_{\pi^*}(x)=\sup_{\pi\in\Pi} \widetilde{J}_\pi(x)
\end{equation}
for all $x\in X.$ It is obvious that the optimal policy $\pi^*$ would depend on $\gamma$ and $\beta.$
Clearly, $\widetilde{J}(x)=\frac1\gamma$ for every $x<0.$

\subsection{Optimality equation and the properties of  its largest minimiser}
Our discounted model with the exponential utility function
reveals some kind of non-stationarity
that is implied by a discount factor. Therefore, one can extend the state space in the following way
$\widetilde{X}:=X\times I,$ with $I:=[\gamma,0)$ (cf. also \cite{ds1}, \cite{br_mor}).
If the process is in the state $(x,\theta)$ and the
insurer selects an action
$a\in A(x),$ then the probability
of moving to a next state $(x',\theta')$ is $q(x'|x,a),$ if $\theta'=\theta\beta$ and is $0,$ if
$\theta'\not=\theta\beta.$
The second component of the state space keeps the track of the  discount factor
that changes over time in a deterministic way.

Furthermore, we can define an {\it extended history-dependent  policy}
$\sigma = (\sigma_{k})_{k\in \N_0},$ where
 $\sigma_{k}$ is a  mapping from the set of extended feasible histories up to the $k$th day
 to the action set $A$ such that
$\sigma_k(h_k) \in A(x_k)$ with
$h_k=(x_0,\theta, a_0,\ldots,a_{k-1},x_k,\theta\beta^k),$ $k\ge 1$ and $h_0=(x_0,\theta).$
Recall that $(x_m,a_m)\in D$ for $m\in\N_0.$
Let $\Xi$ be the set of all extended history-dependent policies.
Note that for any $\sigma\in\Xi$ ($\theta\in I$ is fixed), there exists  a  policy $\pi\in\Pi$  that is
equivalent to $\sigma$ in the following sense:
$$\pi_k(\cdot|\omega_k):=\sigma_k(\cdot|h_k),\quad k\in \N_0.$$
Obviously, $\pi$ depends on $\theta\in I.$
Therefore, for simplicity of notation we shall still use the original policies $\pi\in\Pi,$ and the
expectation operator $\Eop_x^\pi,$
where $x$ is the first component of the initial state.
The dependence on $\theta\in I$ will be denoted by adding
the second variable to the value function.

For any initial state $(x,\theta)\in \widetilde{X}$   we define
$$J(x,\theta):=\inf_{\pi\in\Pi} J_\pi(x,\theta),\quad\mbox{where}\quad
 J_\pi(x,\theta):=\Eop_x^\pi \left(\exp\left\{\theta\sum_{k=0}^{\infty}
\beta^k r(x_k,a_k)\right\}\right)$$
for $\pi\in\Pi.$
 Obviously, $0\le J_\pi(x,\theta)\le 1$ for all $x\in X,$ $\theta\in I$ and $\pi\in \Pi.$
Observe next that our optimisation problem (\ref{optpp}) is equivalent to  the problem of
minimising
$J_\pi(x,\gamma)$ over $\pi\in\Pi.$
By $C(\widetilde{X})$ we denote the space of bounded continuous  real-valued  functions on
$\widetilde{X}.$\\

\begin{theorem}\label{theo:fixed_point}
For every $(x,\theta) \in \widetilde{X}$ the function $J$ is a solution
to the following  discounted optimality equation
\begin{eqnarray}\label{eq:fixedpoint}J(x,\theta)&=&\min_{a\in A(x)}\left[e^{\theta r(x,a)}
\sum_{x'\in X} J(x',\theta\beta)q(x'|x,a)\right]\\
\nonumber &=& \min_{a\in \{0,\ldots,x\}}\left[e^{\theta a}
\left(\sum_{k=a-x}^\infty J(x-a+k,\theta\beta)q_k+ \sum_{k=-\infty}^{a-x-1} q_k\right)\right]
\end{eqnarray}
 \end{theorem}

\begin{proof}
 Clearly, $J(x,\theta)=1$ for $x<0$ and all $\theta\in I.$
Consider the truncated payoff functions
$r_m(x,a)=\min\{m,r(x,a)\}$ with $m\in \N.$
From Proposition 3.1 in \cite{ds1} there is a unique function $w_m\in C(\widetilde{X})$ such that
\begin{equation}
\label{oen}
w_m(x,\theta)=\min_{a\in A(x)}\left[e^{\theta r_m(x,a)}\sum_{x'\in X}
w_m(x',\theta\beta)q(x'|x,a)\right]
\end{equation}
and $w_m(x,\theta)=J_m(x,\theta)$ for all  $(x,\theta)\in \widetilde{X}.$
Here, $J_m(x,\theta)$ denotes the optimal discounted payoff function with $r$ replaced by $r_m$ in
of $J(x,\theta).$
Clearly, the sequence $(w_m(x,\theta))_{m\in N}$ is non-increasing for each $(x,\theta)\in \widetilde{X}.$
Therefore, $\lim_{m\to \infty} w_m(x,\theta)=:w(x,\theta)$ exists.
It is obvious that
$$ w_m(x,\theta)=J_m(x,\theta)\ge J(x,\theta),\quad \mbox{ for }  (x,\theta)\in \widetilde{X}.$$
 Hence,
\begin{equation}
\label{o1}
w(x,\theta)\ge J(x,\theta), \quad \mbox{ for } (x,\theta)\in \widetilde{X}.
\end{equation}
On the other hand, letting $m\to\infty$ in (\ref{oen}), making use of the dominated convergence theorem
and the fact that $A(x)$ is finite for each $x\in X,$ we infer that
\begin{eqnarray}
\label{oeq}
\nonumber
w(x,\theta)&=&\lim_{m\to\infty}\min_{a\in A(x)}\left[e^{\theta r_m(x,a)}\sum_{x'\in X}
w_m(x',\theta\beta)q(x'|x,a)\right]\\
&=&\min_{a\in A(x)}\left[e^{\theta r(x,a)}\sum_{x'\in X}
w(x',\theta\beta)q(x'|x,a)\right]
\end{eqnarray}
for $(x,\theta)\in \widetilde{X}.$ Hence, for any $a\in A(x)$
$$w(x,\theta)\le e^{\theta r(x,a)}\sum_{x'\in X}
w(x',\theta\beta)q(x'|x,a).$$
 Iterating  this inequality $(n-1)$ times
we conclude that
$$w(x,\theta)\le \Eop_x^{\pi} \left(\exp\left\{\theta\sum_{k=0}^{n-1}\beta^k r(x_k,a_k)\right\}
w(x_n,\theta\beta^n)\right).$$
Since $w\le w_n\le 1,$ we have that
$$w(x,\theta)\le \Eop_x^{\pi} \left(\exp\left\{\theta\sum_{k=0}^{n-1}\beta^k r(x_k,a_k)\right\}\right).
$$
Letting $n\to\infty$ and applying the dominated convergence theorem
we have that
$w(x,\theta)\le J_\pi(x,\theta)$ for
$(x,\theta)\in \widetilde{X}.$ Since the policy $\pi$ was chosen arbitrarily, we get that
\begin{equation}
\label{o2}
w(x,\theta)\le J(x,\theta), \quad \mbox{ for } (x,\theta)\in \widetilde{X}.
\end{equation}
Now the assertion follows from (\ref{o1}) and (\ref{o2}).
\end{proof}

\begin{remark}
Among all functions $w$ which satisfy equation \eqref{eq:fixedpoint} and have the property
that $w(x,\theta)=1$ for $x<0$ and $w(x,\theta)\in (0,1]$ for all $(x,\theta)\in \widetilde{X},$
the value function $J$ is the largest solution. This fact follows from the last part of the proof.
\end{remark}

\begin{remark}
Theorem \ref{theo:fixed_point} was proved in the literature for the general state space, weakly continuous transition probabilities
and {\it bounded} costs or rewards \cite{br_mor, ds1}. However, we deal with unbounded payoffs,
therefore we have to truncate them at the level $m$
 and then let $m$ tend to infinity. Such a procedure may have a meaning from the numerical point of view.
\end{remark}

Let us now consider the policy $\pi^+=(g,g\ldots),$ where $g(x)=x^+$
for every $x\in X.$ Hence, this policy  asks the insurer to pay out
everything at each time point until ruin occurs.
Denote by $\vec{\pi}$ the ``1-shifted'' policy for $\pi=(\pi_k)_{k\in\N_0}$, that is,
$\vec{\pi}=(\vec{\pi}_k)_{k\in\N_0},$ where
\begin{equation}
\label{shift_p}
\vec{\pi}_k(\cdot|\omega_k)=\pi_{k+1}(\cdot|x_0,a_0,\omega_k)\quad\mbox{for}\quad  \omega_k\in\Omega_k
\quad\mbox{and}\quad k\in \N_0.
\end{equation}

\begin{lemma}\label{lem:bounds}
For any $x\ge 0$ and $\theta\in I$ the following inequalities hold
$$ e^{\theta x} \underline{h}(\theta)\le  J(x,\theta)\le e^{\theta x}  \overline{h}(\theta),$$
with
$$\underline{h}(\theta):= \prod_{k=1}^\infty \Eop\left(\exp\{\theta\beta^k Z_k^+\}\right)
  \quad\mbox{and}\quad \overline{h}(\theta):=
  \sum_{m=-\infty}^\infty
  \Eop_m^{\pi^+}\left(\exp\left\{\theta\sum_{k=1}^{\tau-1} \beta^k
   x_k\right\}\right)q_m,
 $$
where $x_1=m$ and the empty sum is $0.$
 \end{lemma}

\begin{proof} We start with the upper bound. Since $x\ge 0,$ then $\tau\ge 1.$
For the policy $\pi^+$ we have that
\begin{eqnarray*}
J(x,\theta)&\le& J_{\pi^+}(x,\theta)=\Eop_x^{\pi^+} \left(\exp\left\{\theta\sum_{k=0}^{\tau-1} \beta^k
   x_k\right\}\right)
   =e^{\theta x}\Eop_x^{\pi^+} \left(\exp\left\{\theta\sum_{k=1}^{\tau-1} \beta^k
   x_k\right\}\right)\\
&=& e^{\theta x} \Eop_x^{\pi^+}\left[ \Eop_x^{\pi^+} \left(\exp\left\{\theta\sum_{k=1}^{\tau-1} \beta^k
   x_k\right\}\Big|a_0,x_1\right)\right]\quad(\mbox{under $\pi^+$ we have that  $x_1=Z_1$})\\
&=& e^{\theta x}\Eop_{x}^{\pi^+}  \left[\Eop_{Z_1}^{\vec{\pi}^+}
\left(\exp\left\{\theta\sum_{k=1}^{\tau-1} \beta^k
   x_k\right\}\right)\right]\\
&=& e^{\theta x}
\left(\sum_{m=-\infty}^\infty \Eop_{m}^{\pi^+} \left(\exp\left\{\theta\sum_{k=1}^{\tau-1} \beta^k
   x_k\right\}\right)q_m\right),
\end{eqnarray*}
where in the last equality we make use of the fact that $\Pop(Z_1=m)=q_m$ and  $\vec{\pi}^+=\pi^+.$

On the other hand, the lower bound can be obtained as follows.  First, we claim that for $\pi\in\Pi$ and
$(x,\theta)\in\widetilde{X}$ with $x\ge 0$
\begin{equation}
\label{ind}
e^{\theta x} \prod_{k=1}^{n-1} \Eop\left(\exp\{\theta\beta^k Z_k^+\}\right)
\le
\Eop_{x}^{\pi}\left(\exp\left\{\theta\sum_{k=0}^{n-1}\beta^k r(x_k,a_k)\right\}\right).
\end{equation}
We proceed by induction.  Clearly, $e^{\theta x}\le \min_{a\in A(x)} e^{\theta r(x,a)}.$
Assume now that (\ref{ind}) holds for some $n\ge 1$ and every $x\ge 0$ and $\theta\in I.$
Let $\pi=(\pi_n)_{n\in\N_{0}}$ be any policy. Then, it follows that
 \begin{eqnarray*}
 \lefteqn{
\Eop_{x}^{\pi}\left(\exp\left\{\theta\sum_{k=0}^{n}\beta^k r(x_k,a_k)\right\}\right)
\ge e^{\theta r(x,\pi_0(x))}\Eop_x^{\pi}\left[
\Eop_x^{\pi}\left(\exp\left\{\theta\sum_{k=1}^{n}\beta^k r(x_k,a_k)\right\}\Big|a_0,x_1\right) \right] }\\
 &\ge&
 \min_{a\in A(x)}e^{\theta r(x,a)}\Eop_{x}^{\pi}\left[ 1[x_1\ge 0]\Eop_{x_1}^{\vec{\pi}}
 \left(\exp\left\{\theta\sum_{k=1}^{n}\beta^k r(x_k,a_k)\right\}\right)+1[x_1< 0]\right]\\
 &\ge& \min_{a\in A(x)}e^{\theta a}\Eop_{x}^{\pi}\left[1[x_1\ge 0]e^{\theta\beta x_1}
 \prod_{k=2}^{n} \Eop\left(\exp\{\theta\beta^{k} Z_k^+\}\right)+1[x_1< 0]\right]
 \quad\mbox{(by (\ref{ind}))}\\
 &\ge& \min_{a\in A(x)}e^{\theta a}\Eop_{x}^{\pi}\left[1[x_1\ge 0]e^{\theta\beta x_1}+1[x_1< 0]\right]
 \prod_{k=2}^{n}  \Eop\left(\exp\{\theta\beta^{k} Z_k^+\}\right).
 \end{eqnarray*}
 Furthermore, we have that
  \begin{eqnarray*}
 \lefteqn{\min_{a\in A(x)}e^{\theta a}\Eop_{x}^{\pi}\left[1[x_1\ge 0]e^{\theta\beta x_1}+1[x_1< 0]\right]}\\
 &\ge&
  \min_{a\in A(x)} e^{\theta a} \Big( \sum_{k=a-x}^\infty e^{\theta\beta (x-a+k)}
  q_k+ \sum_{k=-\infty}^{a-x-1}q_k\Big) \\
   &\ge & \min_{a\in A(x)} e^{\theta a} \Big( \sum_{k=1}^\infty e^{\theta\beta (x-a+k)} q_k+
   \sum_{k=-\infty}^{0}e^{\theta\beta (x-a)} q_k\Big) \\
   &= &   \min_{a\in A(x)} e^{\theta a(1-\beta)} \Big( \sum_{k=1}^\infty e^{\theta\beta (x+k)} q_k+
    \sum_{k=-\infty}^{0}e^{ \theta \beta x}q_k\Big) \\
   &=&  e^{\theta x(1-\beta)} e^{\theta\beta x} \Eop e^{\theta\beta Z_1^+}=
   e^{\theta x} \Eop e^{\theta\beta Z_1^+}.
    \end{eqnarray*}
    Hence, we conclude that
  \begin{eqnarray*}
 \Eop_{x}^{\pi}\left(\exp\left\{\theta\sum_{k=0}^{n}\beta^k r(x_k,a_k)\right\}\right) &\ge&
 e^{\theta x} \Eop e^{\theta\beta Z_1^+}
  \prod_{k=2}^{n}  \Eop\left(\exp\{\theta\beta^{k} Z_k^+\}\right)\\&=& e^{\theta x}
   \prod_{k=1}^{n}  \Eop\left(\exp\{\theta\beta^{k} Z_k^+\}\right).
       \end{eqnarray*}
       Therefore, (\ref{ind}) holds for every $n\in \N,$ $x\ge 0,$  $\theta\in I$ and $\pi\in\Pi.$
       Now letting $n\to\infty$ in (\ref{ind})  and making use
       of the dominated convergence theorem we obtain
       the lower bound for $J(x,\theta).$
\end{proof}

\begin{remark}
Note that since $\beta<1,$ we obtain by Jensen's inequality and assumption (A1) that
$$ \underline{h}(\theta) \ge \exp\left(\theta \sum_{k=1}^\infty \beta^k \Eop Z_k^+\right) > 0,\quad \theta\in I.$$
This observation is essential in Theorem \ref{theo:xi_finite}, where we have to take the logarithm of $\underline{h}$.
Also note that $\bar{h}(\theta)\le 1$ since $\theta \in I$.
\end{remark}

Let $x\ge 0.$ For any $\theta \in I$ let us denote
$$ G(x,\theta) := \sum_{k=-x}^\infty J(x+k,\theta) q_k+\sum_{k=-\infty}^{-x-1} q_k,$$
then
\begin{equation}
\label{optg}
 J(x,\theta) = \min_{a\in A(x)} \left[e^{\theta a} G(x-a,\theta\beta)\right],\quad x\ge 0.
 \end{equation}

\begin{lemma}\label{lem:w_decreasing}
The function $J(\cdot,\theta)$ is decreasing for each $\theta\in I$ and
$$ J(x,\theta) \le e^{\theta (x-v)} J(v,\theta)$$ for $ x\ge v\ge 0.$
\end{lemma}

\begin{proof}
Suppose $0\le v< x.$ Then, it follows that
\begin{eqnarray*}
J(x,\theta)&=&\min\left\{ G(x,\theta\beta), \ldots, e^{\theta(x-v-1)}
G(v+1,\theta\beta),\min\left\{e^{\theta (x-v)} G(v,\theta\beta),
\ldots ,e^{\theta (x-0)} G(0,\theta\beta)\right\}\right\}\\
 &  = & \min\left\{ G(x,\theta\beta), \ldots,
 e^{\theta(x-v-1)} G(v+1,\theta\beta),e^{\theta (x-v)} J(v,\theta)\right\}\\
 &  \le & e^{\theta (x-v)} J(v,\theta).
\end{eqnarray*}
Observe that for $v=x-1$ we obtain from the above inequality that $$J(x,\theta)\le e^\theta J(x-1,\theta)
< J(x-1,\theta).$$
 This fact finishes the proof.
\end{proof}

Now let $f^*:\widetilde{X}\mapsto A$ be the  largest minimiser of the right-hand side in
(\ref{optg}) for $x\ge 0$ and let $f^*(x,\theta)=0$ for $x<0$ and all $\theta\in I.$
Then, $f^*(\cdot,\theta)\in \Gamma.$

\begin{lemma}\label{lem:no_dividend_after_payout}
For $x\ge 0$ it holds that $f^*(x-f^*(x,\theta),\theta)=0$ and
$$ J(x,\theta) = e^{\theta f^*(x,\theta)} J(x-f^*(x,\theta),\theta).$$
\end{lemma}

\begin{proof}
By (\ref{optg}) we have that  $J(x-f^*(x,\theta),\theta) \le G(x-f^*(x,\theta),\theta\beta).$
By Lemma \ref{lem:w_decreasing}  (set $v:= x-f^*(x,\theta)$) it follows that
\begin{eqnarray*}
J(x,\theta) &=& e^{\theta f^*(x,\theta)} G(x-f^*(x,\theta),\theta\beta)
\ge e^{\theta f^*(x,\theta)} J(x-f^*(x,\theta),\theta)\\
   &\ge &  J(x,\theta).
\end{eqnarray*}
Thus we have equality and, in particular, $J(x-f^*(x,\theta),\theta)=G(x-f^*(x,\theta),\theta\beta),$
which implies that $a=0$ minimises the expression $e^{\theta a}G(x-f^*(x,\theta)-a,\theta\beta).$
We claim that $a=0$ is the only minimiser of the above expression. Note that, if $f^*(x,\theta)=x,$
then $J(x-f^*(x,\theta),\theta)=J(0,\theta)$ and the result holds true. If, on the other hand,
$f^*(x,\theta)<x,$ then
$$J(x,\theta)=e^{\theta f^*(x,\theta)}G(x-f^*(x,\theta),\theta\beta)<
e^{\theta (f^*(x,\theta)+a)}G(x-f^*(x,\theta)-a,\theta\beta)$$
for $a=1,\ldots,x-f^*(x,\theta).$ This fact implies that $a=0$ is indeed the only minimiser
and, consequently,
 $f^*(x-f^*(x,\theta),\theta)=0$.
\end{proof}

\begin{theorem}\label{theo:xi_finite}
Let $\xi(\theta) := \sup\{x\in\N_0 : \;f^*(x,\theta)=0\}$. Then
$\xi^*:=\sup_{\theta\in I}\xi(\theta)<\infty$ and
$$ f^*(x,\theta) = x-\xi(\theta),\quad \mbox{for all } \;x> \xi(\theta).$$
\end{theorem}

\begin{proof} Fix $\theta\in I$ and let $x\ge 0$ be such that $f^*(x,\theta)=0.$ Note that
such $x$ exists for each $\theta\in I$, because $f^*(0,\theta)=0$ for all $\theta\in I.$
From (\ref{optg}) we obtain that
\begin{equation}
\label{ab}
J(x,\theta)=\sum_{k=-x}^{\infty}J(x+k,\theta\beta)q_k+
\sum_{k=-\infty}^{-x-1}q_k.
\end{equation}
Furthermore, by (\ref{ab}), Lemmas \ref{lem:bounds} and \ref{lem:w_decreasing} we have that
\begin{eqnarray*}
J(x,\theta)&\ge& \sum_{k=1}^\infty J(x+k,\theta\beta)q_k+ \sum_{k=-\infty}^{0} J(x,\theta\beta)q_k\\
&\ge& \sum_{k=1}^\infty e^{\theta\beta (x+k)}\underline{h}(\theta\beta)q_k+\sum_{k=-\infty}^{0}
 e^{\theta\beta x}\underline{h}(\theta\beta)q_k\\
 &=& e^{\theta\beta x} \underline{h}(\theta\beta)\left(\sum_{k=1}^\infty e^{\theta\beta k}q_k+
 \sum_{k=-\infty}^{0}q_k\right)= e^{\theta\beta x} \underline{h}(\theta).
\end{eqnarray*}
Hence,
$$e^{\theta\beta x} \underline{h}(\theta)\le J(x,\theta)\le e^{\theta x}\overline{h}(\theta),$$
which implies that
$$x\le \frac{\ln\overline{h}(\theta)-\ln\underline{h}(\theta)}{\theta(\beta-1)}=:s(\theta).$$
The function $s(\cdot)$ is continuous on $I$ and is finite for each $\theta\in I.$
Additionally,
\begin{eqnarray*}
\lim_{\theta\to 0^-} s(\theta)&=&\lim_{\theta\to 0^-} \left(\frac{\overline{h}'(\theta)}
{\overline{h}(\theta)}-\frac{\underline{h}'(\theta)}{\underline{h}(\theta)} \right)/(\beta-1)\\
&=&
\left(\sum_{k=1}^\infty \beta^k\Eop (Z_k^+) -
\sum_{m=0}^\infty \Eop_m^{\pi^+}\left(\sum_{k=1}^{\tau-1}\beta^k x_k\right)q_m\right)/(1-\beta)<+\infty,
\end{eqnarray*}
which follows by assumption (A1).
Thus, we have shown that $\sup_{\theta\in I}\xi(\theta)<+\infty$. Now let $x> \xi(\theta)$.
We know from Lemma \ref{lem:no_dividend_after_payout} that $f^*(x-f^*(x,\theta),\theta)=0,$
which implies by definition of $\xi(\theta)$ that $f^*(x,\theta)\ge x-\xi(\theta).$
On the other hand, by \eqref{optg} we obtain  that
\begin{eqnarray*}
  J(\xi(\theta),\theta) &\le & e^{\theta (f^*(x,\theta)-(x-\xi(\theta)))}
  G\Big(\xi(\theta)-f^*(x,\theta)+\big(x-\xi(\theta)\big),\theta\beta\Big) \\
   &=& J(x,\theta) e^{-\theta (x-\xi(\theta))} \le J(\xi(\theta),\theta),
\end{eqnarray*}
where the last inequality follows from Lemma \ref{lem:w_decreasing}.
Thus, because $f^*$ is the largest minimiser of (\ref{optg}), we obtain
$$0=f^*(\xi(\theta),\theta) \ge f^*(x,\theta)-(x-\xi(\theta))\ge 0,$$
which implies that $f^*(x,\theta)=x-\xi(\theta).$
\end{proof}

\begin{lemma}\label{lem:band1}
Let $x_0\ge 0$. If $f^*(x_0,\theta)=a_0$ and $f^*(x_0+1,\theta)>0$, then $f^*(x_0+1,\theta)=a_0+1.$
\end{lemma}

\begin{proof}
By definition of $f^*$ we have that
\begin{eqnarray*}
J(x_0,\theta) &=& e^{\theta a_0} G(x_0-a_0,\theta\beta) \left\{ \begin{array}{l}
\le  e^{\theta a} G(x_0-a,\theta\beta), \;\mbox{for}\; a=0,\ldots, a_0\\
                      <  e^{\theta a} G(x_0-a,\theta\beta), \;\mbox{for}\; a=a_0+1,\ldots, x_0,
                     \end{array}\right.
\end{eqnarray*}
which further yields that
\begin{eqnarray*}
  e^{\theta (a_0+1)} G(x_0-a_0,\theta\beta) \left\{ \begin{array}{l}
  \le  e^{\theta (a+1)} G(x_0-a,\theta\beta), \;\mbox{for}\; a=0,\ldots, a_0\\
                      <  e^{\theta (a+1)} G(x_0-a,\theta\beta), \;\mbox{for}\; a=a_0+1,\ldots, x_0.
                     \end{array}\right.
\end{eqnarray*}
Again from the definition of $f^*$ we obtain that
\begin{eqnarray*}
J(x_0+1,\theta) &=& \min_{a\in A(x_0+1)} e^{\theta a} G(x_0+1-a,\theta\beta) \\
  &= & \min\Big\{ G(x_0+1,\theta \beta), \min_{a\in A(x_0)} e^{\theta (a+1)} G(x_0-a,\theta\beta) \Big\}.
\end{eqnarray*}
Since $f^*(x_0+1,\theta)>0,$ it holds that
$$ G(x_0+1,\theta\beta) \ge \min_{a\in A(x_0)} e^{\theta (a+1)} G(x_0-a,\theta\beta).$$
This fact, in turn, together with the previous observation yields by shifting the index that
\begin{eqnarray*}
  e^{\theta (a_0+1)} G(x_0-a_0,\theta\beta) \left\{ \begin{array}{l}
  \le  e^{\theta a} G(x_0+1-a,\theta\beta), \;\mbox{for}\; a=0,\ldots, a_0+1\\
                      <  e^{\theta a} G(x_0+1-a,\theta\beta), \;\mbox{for}\; a=a_0+2,\ldots, x_0+1.
                     \end{array}\right.
\end{eqnarray*}
Thus, it follows that $f^*(x_0+1,\theta)=a_0+1$.
\end{proof}

\subsection{Optimal policy and its structure }

Recall now that $\gamma$ is a constant risk averse coefficient of the insurer.
Consider the following policy $\pi^*:= (\widetilde{g}_{0},\widetilde{g}_{1},\ldots),$
where $\widetilde{g}_{n}(\cdot):=f^*(\cdot,\gamma\beta^{n}).$  We note that $\pi^*\in\Pi^M.$
Clearly,   since $\gamma\in I,$ then $\gamma\beta^n\in I$ for all $n\in\N_0.$

\begin{corollary}\label{ruin_0} Under policy $\pi^*$ the ruin occurs with probability 1, i.e.,
$\Pop^{\pi^*}_x(\tau <+\infty)=1$ for every $x\in X.$
\end{corollary}

\begin{proof}
Assume that the surplus process  equals $x_0\in X.$
If $x_0\ge 0,$ then either $x_0\le \xi(\gamma)$ or $x_0>\xi(\gamma).$
However, from Theorem \ref{theo:xi_finite} we know that in both cases the risk reserve (surplus)
 just after dividend payment is
always less or equal to $\xi(\gamma)\le\xi^*.$ Therefore, the ruin will occur, if
there appears a sequence of length $\xi^*+1$ of negative incomes. But the probability
 that such a sequence appears, equals
\begin{equation}
\label{compr}
\Pop(Z_{1}<0,\ldots,Z_{\xi^*+1}<0)=\left(\sum_{m=-\infty}^{-1}q_m\right)^{\xi^*+1},
\end{equation}
which is positive by (A2). If the ruin has not occurred up to the $l$th day, where  $l:=\xi^*+1,$
 then again $x_l\le \xi(\gamma\beta^l)$ or $x_l>\xi(\gamma\beta^l).$
But  from Theorem \ref{theo:xi_finite}
in both cases the risk reserve just after dividend payment is
always less or equal to $\xi(\gamma\beta^l)\le\xi^*.$   The probability that there exists
a sequence of length $\xi^*+1$ of negative incomes is $\left(\sum_{m=-\infty}^{-1}q_m\right)^{\xi^*+1}.$
Thus, considering the states $x_{k(\xi^*+1)},$ $k\in\N_0,$ we may
define the following events
$$A_k=\{Z_{k(\xi^*+1)+1}<0,Z_{k(\xi^*+1)+2}<0\ldots,Z_{(k+1)(\xi^*+1)}<0\},
\quad k\in\N_0.$$
By the second Borel-Cantelli lemma $\Pop(A_k \;i.o.)=1.$ Therefore, the ruin must occur.
\end{proof}

\begin{theorem}\label{theo:optimal}
The Markov policy $\pi^*$ is optimal, i.e.,
$$\widetilde{J}(x)=\widetilde{J}_{\pi^*}(x)=\frac 1\gamma J_{\pi^*}(x,\gamma)=\frac1\gamma J(x,\gamma)$$
for $x\in X.$
\end{theorem}

\begin{proof} From Theorem \ref{theo:fixed_point} and the definition of $\pi^*$ we obtain
for every $x\in X$ that
\begin{eqnarray*}
J(x,\gamma)&=&\min_{a\in A(x)}\left[e^{\gamma r(x,a)}\sum_{x'\in X} J(x',\gamma\beta)q(x'|x,a)\right]\\
&=&e^{\gamma \widetilde{g}_0(x)}\sum_{x'\in X} J(x',\gamma\beta)q(x'|x,\widetilde{g}_0(x)).
\end{eqnarray*}
Assume that $x\ge 0.$
Iterating the last equality $n$ times under the Markov policy $\pi^*,$ we obtain that
\begin{equation}
\label{opt_pol}
 J(x,\gamma) =
\Eop_x^{\pi^*}\left(\exp\left\{\gamma\sum_{k=0}^{(\tau-1)\wedge (n-1)} \beta^k a_k\right\}
 \left(J(x_n,\gamma\beta^n) 1[\tau\ge n] + 1[\tau < n] \right)\right).
 \end{equation}
Observe now that
$$0\le \Eop_x^{\pi^*}\left(\exp\left\{\gamma\sum_{k=0}^{(\tau-1)\wedge (n-1)} \beta^k a_k\right\}
J(x_n,\gamma\beta^n) 1[\tau\ge n]\right) \le \Eop_x^{\pi^*}1[\tau\ge n]= \Pop_x^{\pi^*}(\tau\ge n).$$
But by  Corollary \ref{ruin_0},  $\Pop_x^{\pi^*}(\tau\ge n)\to 0$ as $n\to\infty.$ Hence,
letting $n\to\infty$ in (\ref{opt_pol}) and making use of the dominated convergence theorem we obtain that
$$\inf_{\pi\in\Pi}J_{\pi}(x,\gamma)=J(x,\gamma) =
\Eop_x^{\pi^*}\left(\exp\left\{\gamma\sum_{k=0}^{\tau-1} \beta^k a_k\right\}
\right)=J_{\pi^*}(x,\gamma)$$
for $x\in X.$ The conclusion follows by multiplying the above display by the number $1/\gamma.$
\end{proof}

\begin{definition}
A function $g\in \Gamma$ is called a {\em band-function}, if there exists numbers $n\in\N_0$
and $c_0,\ldots,c_n,d_1,\ldots,d_n\in\N_0$ such that $d_k-c_{k-1}\ge 2$ for $k=1,\ldots,n,$
$0\le c_0\le d_1\le c_1\le d_2\le\ldots\le d_n\le c_n$ and
$$g(x)=    \left\{ \begin{array}{l}
  0, \mbox{ if } x\le c_0\\
x-c_k, \mbox{ if } c_k<x<  d_{k+1}\\
  0, \mbox{ if } d_k\le x\le c_k\\
x-c_n, \mbox{ if } x> c_{n}.
\end{array}\right.
$$
A Markov policy $\pi=(g_m)_{m\in N_0}$ is called a band-policy, if  $g_m$ is a band-function for every
$n\in \N_0.$
\end{definition}

\begin{theorem}\label{theo:band1}
The optimal Markov policy $\pi^*$ is a band-policy.
\end{theorem}

\begin{proof}
Recall that $\widetilde{g}_n(\cdot)=f^*(\cdot,\gamma\beta^n)$ for $n\in\N_0.$
By Theorem \ref{theo:xi_finite} we have $\widetilde{g}_n(x)= x-\xi(\gamma\beta^n)$ for all
$x> \xi(\gamma\beta^n)$. For $x\le\xi(\gamma\beta^n)$ we have to distinguish different cases. If
$\widetilde{g}_n(x)=0$ for all $x=0,\ldots ,\xi(\gamma\beta^n)$, then clearly $\widetilde{g}_n$ is a
band-function. If there exists an $0<x_0\le \xi(\gamma\beta^n)$ such that $\widetilde{g}_n(x)=0$
for $x=0,\ldots ,x_0-1$ and $\widetilde{g}_n(x_0)>0$, then by Lemma
\ref{lem:band1}
$\widetilde{g}_n(x_0)=1$. If further $\widetilde{g}_n(x_0+m)>0$ for
$m=1,\ldots ,\xi(\gamma)-x_0-1$ then by induction
$$ \widetilde{g}_n(x_0+m)=\widetilde{g}_n(x_0+m-1)+1=\ldots = \widetilde{g}_n(x_0)+m =m+1.$$
If $\widetilde{g}_n(x_0+1)=0$ we either have $\widetilde{g}_n(x)=0$ for $x=x_0+1,\ldots,\xi(\gamma\beta^n)$
or there exists an
$x_1$ such that  $x_0<x_1\le \xi(\gamma\beta^n)$ and
$\widetilde{g}_n(x_0+m)=0$ for $m=1,\ldots,x_1-x_0-1$ and $\widetilde{g}_n(x_1)>0$.
Now we proceed in the same way as with $x_0$. After a finite number
of steps we reach $\xi(\gamma\beta^n)$. In any case $\widetilde{g}_n$ is a
band-function.
\end{proof}

\begin{remark}
In the risk neutral dividend payout problem, the optimal policy is a stationary band-policy, i.e.,
it consists of the same band-function at all time points. The risk neutral problem can formally
be obtained as a limit $\lim_{\gamma \to 0} \frac1\gamma (e^{\gamma x}-1)$. Hence, the exponential utility function
only implies some kind of non-stationarity of the optimal policy and
 thus does not really make  it necessary to consider history-dependent policies.
\end{remark}


\section{The Power Utility Function}

In this section we assume that the insurer is equipped with the power utility function
$U_\gamma(x)= x^\gamma,$ where $\gamma $ is a fixed number from the interval $(0,1).$
The decision maker wishes to maximise the following expected discounted payoff
$$  \widehat{J}_\pi(x):= \Eop_x^\pi U_\gamma\left(\sum_{k=0}^{\infty}\beta^k r(x_k,a_k)\right)$$
and to find a policy $\pi^*\in\Pi$ (if exists) such that
\begin{equation}
\label{oppp}
\widehat{J}(x):=\sup_{\pi\in\Pi} \widehat{J}_\pi(x)=\widehat{J}_{\pi^*}(x)
\end{equation}
for all $x\in X.$  Clearly,  $\widehat{J}(x)=0$ for $x<0.$ In Lemma \ref{lem:pbounds}
we show that under our assumptions
$\widehat{J}(x)<+\infty$ for each $x\ge 0.$ Moreover, note that for $x\ge 0$
$$  \widehat{J}(x):= \sup_{\pi\in\Pi} \Eop_x^\pi U_\gamma\left(\sum_{k=0}^{\tau-1}\beta^k a_k\right).$$

 \subsection{Optimality equation and the properties of its largest maximiser} Contrary to the exponential utility function,
 the power utility function reveals certain non-separability, that is implied by the fact that
 the expectation operator is only linear.  Therefore, we again extend the state space
 by defining the new state space $\widehat{X}:=X\times [0,\infty)$ (cf. \cite{br_mor}). In this case, the second component
 is responsible for the accumulated payoffs so far.
 If the process is in the state $(x,y)$ and the insurer selects an action $a\in A(x),$
 then the probability of moving to a next state $(x',y')$ is $q(x'|x,a),$
 if $y' = \frac{y+a}\beta$ and is $0,$
 if $y'\not=\frac{y+a}\beta.$ Hence, we can observe that the second component is again established in a
 deterministic way, but it differs from the previous case, since $y'$ depends on the action chosen by the
 insurer.

Let us define the feasible extended histories of the process  up to the $k$th day as follows
\footnote{We use the same symbol $h_k$ as in the previous section to denote an extended feasible history
of the process up to the $k$th day. But there is no confusion, since $h_k$
in this subsection refers only to the power
utility case. The same remark applies to the policy $\sigma,$  the set $\Xi$ and the functions $J$ and $J_\pi$ defined below.}
$$h_0=(x_0,y_0)\quad\mbox{and}\quad h_k=(x_0,y_0, a_0,x_1,y_1,\ldots,a_{k-1},x_k,y_k),\;\;k\ge 1,$$
where $(x_m,a_m)\in D$ for each $m\in\N_0$ and with $y_{m+1}$ given by the recurrence equation
$$y_{m+1}:=\frac{y_{m}+a_{m}}\beta,\quad m\in \N_0.$$
 Then, we can define, as usual, an {\it extended history-dependent  policy}
 $\sigma = (\sigma_{k})_{k\in \N_0},$ where
 $\sigma_{k}$ is a mapping from the set of feasible extended histories
 up to the $k$th day to the action set $A$ such that
$\sigma_k(h_k) \in A(x_k)$ with  $h_k$ defined above.
Let $\Xi$ be the set of all such policies.
Note that for any $\sigma\in\Xi$ ($y\ge 0$ is fixed), there exists  a  policy $\pi\in\Pi$  that is
equivalent to $\sigma$ in the following sense:
$$\pi_k(\cdot|\omega_k):=\sigma_k(\cdot|h_k),\quad\omega_k\in\Omega_k,\quad k\in \N_0.$$
Obviously, $\pi$ must depend on $y.$
Thus again, for simplicity of notation we shall still use the original set of policies $\Pi,$
and the expectation operator $\Eop_x^\pi,$
where $x$ is the first component of the initial state.
The dependence on $y\ge 0$ of a policy will be indicated by writing
the second variable to the value function.

In what follows, we put for $n\in \N,$ $\pi\in\Pi,$ $x\ge 0$ and $y\ge 0$
$$
J_{n,\pi}(x,y) :=   \Eop_x^\pi \left(\sum_{k=0}^{(n-1)\wedge(\tau-1)}\beta^k r(x_k,a_k)+y\right)^\gamma=
\Eop_x^\pi \left(\sum_{k=0}^{(n-1)\wedge(\tau-1)}\beta^k a_k+y\right)^\gamma$$
and
$$ J_{n}(x,y) := \sup_{\pi\in\Pi} J_{n,\pi}(x,y).$$
Moreover, for $\pi\in\Pi,$ $x\ge 0$ and $y\ge 0$ we set
$$J(x,y)=\sup_{\pi\in\Pi}J_\pi(x,y), \quad\mbox{where}\quad J_\pi(x,y)=
\Eop_x^\pi \left(\sum_{k=0}^{\tau-1}\beta^k a_k+y\right)^\gamma.$$
If $x<0,$ then $J(x,y)=y^\gamma$ for $y\ge 0.$
Obviously, $ J(x,0) = \widehat{J}(x).$

Before we formulate our first result, we introduce a specific subset of policies $\widehat{\Pi}\subset\Pi.$
Let $F$ be the set of functions $f:\widehat{X}\mapsto A$ such that
$f(x,y)\in A(x)$ for all $y\ge 0$ and let
$(f_k)_{k\in\N_0}$ be a sequence of functions with $f_k\in F.$
Then,  $\widehat{\Pi}$  is the set of all policies $\pi=(\pi_k)_{k\in\N_0}$ defined in the following way
$$ \pi_k(\omega_k):=f_k(x_k,y_k),\quad \omega_k\in\Omega_k,\quad k\in\N_0,
$$
where $y_k:=\frac{y_{k-1}+a_{k-1}}\beta,$ $k\in\N,$  and $y_0:=y\ge 0$ is a fixed number.
Furthermore, we shall identify a policy $\pi\in\widehat{\Pi}$ with the sequence $(f_k)_{k\in\N_0}$ by writing
$\pi=(f_k)_{k\in\N_0}.$

Next for any  function $h:\widehat{X}\mapsto \R$ we define an  operator $T$  as follows
\begin{equation}
\label{T_op}
Th(x,y):=\beta^\gamma \max_{a\in A(x)} \left[ \sum_{x'\in X} h\left(x',\frac{a+y}\beta\right) q(x'|x,a)  \right].
\end{equation}
Let $f\in F$ be the maximiser of the right-hand side in (\ref{T_op}), i.e., $f(x)$ attains the maximum on the right-hand side
of (\ref{T_op}) for all $x\in X$. We also set
$$T_{f}h(x,y):=\beta^\gamma  \sum_{x'\in X} h\left(x',\frac{f(x,y)+y}\beta\right) q(x'|x,f(x,y)).$$
Note that $T_{f}h=Th.$

\begin{theorem}\label{theo:Jprecursive}
For each $n\in\N_0$ the value function $J_n$ satisfies the equation
\begin{equation}
\label{recur}
J_{n+1} = TJ_{n}
\end{equation}
with $J_n(x,y)=y^\gamma$ for $x<0$  and  $J_0(x,y)= y^\gamma$.
Let $\bar{f}_{l}\in F$ be such that $J_{l+1} =T_{\bar{f}_{l}}J_{l}$ for $l=0,\ldots,n.$
Then, $\bar{\pi}=(\bar{f}_{n},\ldots,\bar{f}_{0})$
is optimal for $J_{n+1},$ i.e., $J_{n+1}=J_{n+1,\bar{\pi}}.$
\end{theorem}

\begin{proof}
Let $n=1.$ Then, by the definition of $J_1$ we have that
$$J_1(x,y)=\sup_{\pi\in\Pi}\Eop_x^\pi(\beta^0 a_0+y)^\gamma=\max_{a\in A(x)}(a+y)^\gamma=(x+y)^\gamma.$$
On the other hand,
 $$TJ_0(x,y)=\beta^\gamma \max_{a\in A(x)}\left(\frac{a+y}\beta\right)^\gamma
=(x+y)^\gamma=T_{\bar{f}_{0}}J_{0}(x,y)$$ with $\bar{f}_0(x,y)=x.$

 Assume now that $J_l=TJ_{l-1}=T_{\bar{f}_{l-1}}J_{l-1}$ with $\bar{f}_{l-1}\in F$  for all $l=1,\ldots,n$
 and let $(\bar{f}_{n-1},\ldots, \bar{f}_{0})$
 be an optimal policy for $J_n.$
We show that $J_{n+1}=TJ_{n}=T_{\bar{f}_{n}}J_{n}$ and $(\bar{f}_{n},\ldots, \bar{f}_{0})$ is optimal for
$J_n.$

We have that
$$J_{n+1}(x,y)=\sup_{\pi\in\Pi} \Eop_x^\pi \left(\sum_{k=1}^{n\wedge(\tau-1)}\beta^k
a_k+a_0+y\right)^\gamma$$
with the convention that the empty sum equals $0$ (when $\tau=1$).
 Recall that $\vec{\pi}$ denotes a ``1-shifted'' policy, see (\ref{shift_p}).
We further get the following
\begin{eqnarray}
\label{nr1}
J_{n+1}(x,y)
&=&\beta^\gamma\sup_{\pi\in\Pi} \Eop_x^\pi \left(\sum_{k=1}^{n\wedge(\tau-1)}\beta^{k-1} a_k+
\frac{a_0+y}\beta\right)^\gamma\\\nonumber
&=&\beta^\gamma\sup_{\pi\in\Pi} \Eop_x^\pi \left[
\Eop_x^\pi \left(\left(\sum_{k=1}^{n\wedge(\tau-1)}\beta^{k-1} a_k+
\frac{a_0+y}\beta\right)^\gamma\Big|a_0,x_1\right)\right]\\\nonumber
&\le&\beta^\gamma\sup_{a_0\in A(x)}\left[ \sum_{m=-\infty}^\infty
\sup_{\vec{\pi}\in\Pi}\Eop_{x-a_0+m}^{\vec{\pi}} \left(\sum_{k=1}^{n\wedge(\tau-1)}\beta^{k-1} a_k+
\frac{a_0+y}\beta\right)^\gamma q_m\right]\\\nonumber
&=& \beta^\gamma\max_{a\in A(x)}\left[\sum_{m=-\infty}^\infty
J_{n}\left(x-a+m,\frac{a+y}\beta\right) q_m\right]=TJ_n(x,y).
\end{eqnarray}
Let $\bar{f}_n$ be such  that $TJ_n(x,y)=T_{\bar{f}_n}J_n(x,y).$
Put $\bar{\pi}=(\bar{f}_n,\bar{f}_{n-1},\ldots,\bar{f}_0).$ By induction
assumption
$\vec{\bar{\pi}}=(\bar{f}_{n-1},\ldots,\bar{f}_0)$ is optimal for $J_n.$
Hence, it follows that
\begin{eqnarray}
\label{nr2}
TJ_n(x,y)&=&\beta^\gamma \max_{a\in A(x)} \left[ \sum_{m=-\infty}^\infty J_{n\vec{\bar{\pi}}}\left(x-a+m,
\frac{a+y}\beta\right) q_m  \right]\\ \nonumber
&=&\max_{a\in A(x)} \beta^\gamma\sum_{m=-\infty}^\infty
\Eop_{x-a+m}^{\vec{\bar{\pi}}}\left(\sum_{k=0}^{(n-1)\wedge(\tau-1)}\beta^ka_{k+1}+
\frac{y+a}\beta\right)^\gamma q_m\\\nonumber
&=& \max_{a\in A(x)}\sum_{m=-\infty}^\infty
\Eop_{x-a+m}^{\vec{\bar{\pi}}}\left(\sum_{k=0}^{(n-1)\wedge(\tau-1)}\beta^{k+1}a_{k+1}+y+a\right)^\gamma
q_m\\\nonumber
&\le&
\sup_{\pi\in\Pi}\Eop_{x}^{\pi}\left(\sum_{k=0}^{n\wedge(\tau-1)}\beta^{k}a_{k}+y\right)^\gamma=J_{n+1}(x,y).
\end{eqnarray}
Thus, (\ref{nr1}) and (\ref{nr2}) yield that $J_{n+1}=TJ_n.$
The fact that $\bar{\pi}$ defined above is optimal for $J_{n+1}$
follows from repeating the calculations  in (\ref{nr2}) applied to $T_{\bar{f}_n}J_n.$
\end{proof}

The next result  can be concluded from Theorem \ref{theo:Jprecursive}.

\begin{theorem}\label{theo:Jfixedpoint}
The function $J$ satisfies the following equation
\begin{eqnarray}
\label{optp}
 J(x,y) &=& \beta^\gamma \max_{a\in A(x)}
\left[ \sum_{x'\in X} J\left(x', \frac{a+y}\beta\right) q(x'|x,a) \right]\\\nonumber
&=&\beta^\gamma \max_{a\in A(x)}
\left[ \sum_{k=-\infty}^\infty J\left(x-a+k, \frac{a+y}\beta\right) q_k \right]
\end{eqnarray}
 for $x\in\N_0$ and $ y \ge 0.$
\end{theorem}

\begin{proof}
It is obvious that the sequence of functions $(J_n(x,y))_{n\in\N_0}$ is increasing for each $(x,y)\in\widehat{X}.$
Hence, $w(x,y):=\lim_{n\to\infty} J_n(x,y)$
exists for every $(x,y)\in \widehat{X}.$ Since $J_n\le J,$ then it follows that $w\le J.$ On the other hand,
for any policy $\pi\in\Pi$ we obtain that
$J_{n,\pi}\le J_n.$ Letting $n\to\infty,$ making use of the monotone convergence theorem
and  taking the supremum over $\pi\in\Pi,$ it is easily seen that
$$J(x,y)=\sup_{\pi\in\Pi} \Eop_x^\pi\left(\sum_{k=0}^{\tau-1}\beta^ka_k+y\right)^\gamma
\le w(x,y).$$ Consequently, $w=J.$ Equation (\ref{optp}) follows from (\ref{recur}) by letting
$n\to\infty$ and replacing the maximum with the limit.
\end{proof}

\begin{remark}
The counterpart of Theorem  \ref{theo:Jfixedpoint} is Theorem 4.1(a) in \cite{br_mor}. However, again as in the exponential case
this theorem was proved for general state space, weakly continuous transitions and {\it bounded} costs.
\end{remark}

The next lemma provides the following bounds for the function $J.$

\begin{lemma}\label{lem:pbounds}
For any $x\in\N_0$ and $y\ge 0$ it follows that
$$ (x+y)^\gamma \le J(x,y) \le \left(x+\beta \frac{\Eop Z^+}{1-\beta}+y\right)^\gamma.$$
\end{lemma}

\begin{proof} Let $\pi$ be a policy such that $\pi_0(\cdot|\omega_0)=x$ and
$\pi_k(\cdot|\omega_k)=0$ for $k\ge 1.$ Then,  $J(x,y)\ge J_{\pi}(x,y)=(x+y)^\gamma.$
The upper bound for the function $J$ is due to the Jensen inequality and Theorem 9.2.3(a)
in \cite{br} that gives  the upper bound for the risk neutral setting.
\end{proof}

For simplicity for  any $x\in\N_0$  and $y\ge 0$ we define
$$ G(x,y) :=  \sum_{k=-\infty}^\infty J(x+k, y) q_k.$$
From (\ref{optp}) we obtain that
\begin{equation}
\label{optpg}
 J(x,y) = \beta^\gamma \max_{a\in A(x)}
G\left(x-a, \frac{a+y}\beta\right), \quad x\in \N_0,\;\; y\ge 0.
\end{equation}

\begin{lemma}\label{theo:propJ}
For all $0\le v \le x$ we have that $J(x,y) \ge J\left(x-v,y+v\right)$.
\end{lemma}

\begin{proof}
 Suppose  that  $0< v \le x.$ For part (b) observe that
\begin{eqnarray*}
 J(x,y) &=& \max\left\{\left\{\beta^\gamma G\left(x,\frac{y}\beta\right), \ldots,
  \beta^\gamma G\left(x-v+1,\frac{y+v-1}\beta\right)\right\},\right.\\
  &&\left.
  \max\left\{\beta^{\gamma} G\left(x-v,\frac{y+v}\beta\right),\ldots,\beta^{\gamma}
  G\left(0,\frac{y+x}\beta\right) \right\}  \right\}\\
   &=& \max\left\{\left\{\beta^\gamma G\left(x,\frac{y}\beta\right), \ldots,
  \beta^\gamma G\left(x-v+1,\frac{y+v-1}\beta\right)\right\},\right.\\
  &&\left. \beta^\gamma\max_{a\in A(x-v)}G\left(x-v-a,\frac{y+v+a}\beta\right)  \right\}\ge
  J\left(x-v,y+v\right),
  \end{eqnarray*}
where the last inequality is due to (\ref{optpg}).
\end{proof}

In what follows let $f^*\in F$ be the largest maximiser of the right-hand side in (\ref{optpg}).
For completeness, set $f^*(x,y)=0$ for $x<0.$

\begin{lemma}\label{lem:paydown}
For $x\in\N_0$ and $y\ge 0$ it follows that $f^*\big( x-f^*(x,y),y+f^*(x,y)\big)=0$.
\end{lemma}

\begin{proof}
By (\ref{optpg}) we obtain $J(x,y) \ge \beta^\gamma G\left(x,\frac{y}\beta\right),$ which implies that
$$
  J(x,y) = \beta^\gamma G\left(x-f^*(x,y), \frac{y+f^*(x,y)}{\beta}\right)
  \le J\left( x-f^*(x,y),y+f^*(x,y)\right)\le J(x,y),
$$
where the second inequality follows from Lemma \ref{theo:propJ}. Hence,
$$\beta^\gamma G\left(x-f^*(x,y), \frac{y+f^*(x,y)}{\beta}\right)
= J\left( x-f^*(x,y),y+f^*(x,y)\right),$$
which implies that $a=0$ maximises the expression
$\beta^{\gamma}G(x-f^*(x,y)-a,\frac{y+f^*(x,y)+a}\beta).$
We claim that $a=0$ is the only maximiser of this expression. Obviously, if $f^*(x,y)=x,$
then the result follows. If, on the other hand,
$f^*(x,y)<x,$ then
$$J(x,y)=\beta^{\gamma}G(x-f^*(x,y),\frac{y+f^*(x,y)}\beta)>
\beta^{\gamma}G(x-f^*(x,y)-a,\frac{y+f^*(x,y)+a}\beta)$$
for $a=1,\ldots,x-f^*(x,y).$ This fact, in turn, implies that $a=0$ is the only maximiser,
which concludes the proof.
\end{proof}

\begin{lemma}\label{lem:pzero}
Let $\xi(y) := \sup\{ x\in\N_0 : f^*(x,y) =0\}$. Then $\xi^*:=\sup_{y\ge 0} \xi(y) < \infty.$
\end{lemma}

\begin{proof} Fix $y\ge 0.$ Let $x\in \N_0$ be such that
$f^*(x,y)=0.$ Clearly, such $x\in\N_0$ exists.  From (\ref{optg}) we have that
\begin{eqnarray*}
  J(x,y) &=& \beta^\gamma\left( \sum_{k=-x}^\infty J\big(x+k, \frac{y}\beta\big) q_k +
   \Big(\frac{y}\beta\Big)^\gamma \sum_{k=-\infty}^{-x-1} q_k \right)\\
   &\le& \beta^\gamma\left(  \sum_{k=-x}^\infty \big(x+k +\frac{\beta}{1-\beta} \Eop Z^+ +
   \frac{y}\beta\big)^\gamma q_k + \Big(\frac{y}\beta\Big)^\gamma \sum_{k=-\infty}^{-x-1} q_k\right)  \quad
   \mbox{(by Lemma \ref{lem:pbounds})}\\
   &=&  \sum_{k=-x}^\infty \big(\beta x+\beta k +\frac{\beta^2}{1-\beta} \Eop Z^+ + y\big)^\gamma q_k +
   {y}^\gamma \sum_{k=-\infty}^{-x-1} q_k \\
   &\le&  \left( \sum_{k=-x}^\infty \big(\beta x+\beta k +\frac{\beta^2}{1-\beta} \Eop Z^+ + y\big) q_k +
   {y} \sum_{k=-\infty}^{-x-1} q_k\right)^\gamma \quad\mbox{(by the Jensen inequality)}\\
   &\le& \left( \beta x+\beta \Eop Z^+ +\frac{\beta^2}{1-\beta} \Eop Z^+ + y\right)^\gamma
   =  \left( \beta x+\frac{\beta}{1-\beta} \Eop Z^+ +{y}\right)^\gamma.
\end{eqnarray*}
On the other hand, making use again of Lemma \ref{lem:pbounds} we have that
$J(x,y) \ge (x+y)^\gamma$ and, consequently,
$$ (x+y)^\gamma \le \left( \beta x+\frac{\beta}{1-\beta} \Eop Z^+ +y\right)^\gamma$$ if and only if $x\le
\frac{\beta}{(1-\beta)^2} \Eop Z^+,$ which is independent of $y$ and implies the result.
\end{proof}

The next result is a counterpart of Lemma \ref{lem:band1} and provides further properties of the function $f^*\in F.$

\begin{lemma}\label{lem:band2}
Let $x_0\in\N_0$ and $y_0\ge 1$. If $f^*(x_0,y_0)=a_0$ and $f^*(x_0+1,y_0-1)>0$, then
$f^*(x_0+1,y_0-1)=a_0+1.$
\end{lemma}

\begin{proof}
By definition of $f^*$ and (\ref{optp}) we have that
\begin{eqnarray*}
J(x_0,y_0) &=& \beta^{\gamma} G\left(x_0-a_0,\frac{y_0+a_0}\beta\right) \left\{ \begin{array}{l}
\ge  \beta^{\gamma} G\left(x_0-a,\frac{y_0+a}\beta\right), \;\mbox{for}\; a=0,\ldots, a_0\\
                      >  \beta^{\gamma} G\left(x_0-a,\frac{y_0+a}\beta\right), \;\mbox{for}\;
                      a=a_0+1,\ldots, x_0.
                     \end{array}\right.
\end{eqnarray*}
The above display implies that
\begin{eqnarray}
\label{shin}
  \beta^{\gamma} G\left(x_0-a_0,\frac{y_0+a_0}\beta\right) \left\{ \begin{array}{l}
  \ge  \beta^{\gamma} G\left(x_0-(a-1),\frac{y_0+a-1}\beta\right), \;\mbox{for}\; a=1,\ldots, a_0+1\\
                      > \beta^{\gamma} G\left(x_0-(a-1),\frac{y_0+a-1}\beta\right), \;\mbox{for}\;
                      a=a_0+2,\ldots, x_0+1.
                     \end{array}\right.
\end{eqnarray}
On the other hand, we also obtain that
\begin{eqnarray*}
J(x_0+1,y_0-1) &=& \max_{a\in A(x_0+1)} \beta^{\gamma} G\left(x_0+1-a,\frac{y_0-1+a}\beta\right) \\
  &= & \max\left\{  \beta^{\gamma} G\left(x_0+1,\frac{y_0-1}\beta\right),
  \max_{a\in A(x_0)} \beta^{\gamma} G\left(x_0-a,\frac{y_0+a}\beta\right) \right\}.
\end{eqnarray*}
Since $f^*(x_0+1,y_0-1)>0,$ we infer that
$$  \beta^{\gamma} G\left(x_0+1,\frac{y_0-1}\beta\right) \le
\max_{a\in A(x_0)} \beta^{\gamma}G\left(x_0-a,\frac{y_0+a}\beta\right).$$
This fact and (\ref{shin}) yield that
$$   \beta^{\gamma} G\left(x_0+1-(a_0+1),\frac{y_0-1+(a_0+1)}\beta\right)
  \ge  \beta^{\gamma} G\left(x_0+1-a,\frac{y_0-1+a}\beta\right),
  $$ for $a=0,\ldots, a_0+1$ and
   $$\beta^{\gamma} G\left(x_0+1-(a_0+1),\frac{y_0-1+(a_0+1)}\beta\right)  >
    \beta^{\gamma} G\left(x_0+1-a,\frac{y_0-1+a}\beta\right),$$
    for $a=a_0+2,\ldots, x_0+1.$
Thus, it follows that $f^*(x_0+1,y_0-1)=a_0+1$.
\end{proof}

\subsection{Optimal policy}
Let $y_0\ge 0$ be fixed and let $x_0\in X$ be the initial state.  Consider the following
 policy $\pi^* := (\pi^*_k)_{k\in\N_0}$ generated by $f^*$ in the following way
 \begin{equation}
\label{s_pol}
 \pi^*_k(\omega_k):=f^*(x_k,y_k),\quad \omega_k\in\Omega_k\quad
k\in\N_0,
\end{equation}
where
 \begin{equation}
\label{s_pol_y}
y_k:=\frac{y_{k-1}+f^*(x_{k-1},y_{k-1})}\beta=\frac{y_0+\sum_{m=0}^{k-1}\beta^m f^*(x_m,y_m)}{\beta^k},
\quad k\in\N.
\end{equation}
Obviously, $\pi\in\widehat{\Pi}.$

\begin{corollary}\label{ruin_p}
Under policy $\pi^*$ ruin occurs with probability 1.
\end{corollary}

\begin{proof}
We proceed along similar lines as in the proof of Corollary \ref{ruin_0}. Let $(x_0,y_0)\in\widehat{X}$ with $x_0\ge 0.$
Then, either $x_0\le\xi^*$ or $x_0>\xi^*.$ Observe that in both cases the risk reserve just after the dividend payment
is less or equal to $\xi^*.$ The first case is obvious. In the second case, we deduce from Lemma \ref{lem:paydown} that
$f^*(x_0-f^*(x_0,y_0),y_0+f^*(x_0,y_0))=0,$ which means by Lemma \ref{lem:pzero}
that $x_0-f^*(x_0,y_0)\le\xi(y_0+f^*(x_0,y_0))\le\xi^*.$
Hence, the ruin occurs, if there appears a sequence of length $\xi^*+1$ of negative incomes. The probability
of such event equals $\left(\sum_{m=-\infty}^{-1}q_m\right)^{\xi^*+1},$ see also (\ref{compr}). If the ruin has not
occurred up to the $l$th day with $l=\xi^*+1,$ then either  $x_l\le\xi^*$ or $x_l>\xi^*.$ Now the remaining part follows
from the proof of Corollary  \ref{ruin_0}.
\end{proof}

\begin{theorem}\label{optimalp}
 For every $(x,y)\in\widehat{ X}$ it holds that $J(x,y) = J_{\pi^*}(x,y).$\end{theorem}

\begin{proof}
From Theorem \ref{theo:Jfixedpoint}  and the definition of $\pi^*$ given in (\ref{s_pol}) we have that
$$ J(x,y)=\beta^\gamma\sum_{x_1\in X} J(x_1,y_1)q(x_1|x,f^*(x,y))$$
for $x\in\N_0$ and $y\ge 0$ with $y_1$ defined in (\ref{s_pol_y}) and $y_0:=y.$
Hence,
 \begin{eqnarray}
 \label{pomoc}
 J(x,y)&=&\beta^\gamma\sum_{x_1\ge 0} J(x_1,y_1)q(x_1|x,f^*(x,y))
 +\beta^\gamma\sum_{x_1<0} (y_1)^\gamma q(x_1|x,f^*(x,y))\\\nonumber
 &=&\beta^\gamma\sum_{x_1\ge 0} J(x_1,y_1)q(x_1|x,f^*(x,y))+\sum_{x_1<0} (y+f^*(x,y))^\gamma
 q(x_1|x,f^*(x,y)).
 \end{eqnarray}
 Iterating (\ref{pomoc}) $(n-1)$ times and making use of the policy $\pi^*$ we arrive at the following
 equation
\begin{eqnarray}
\label{eq1}
\lefteqn{ J(x,y) =(\beta^\gamma)^n  \Eop_x^{\pi^*}\left(J(x_n,y_n)1[\tau>n]\right)+}\\ \nonumber
&&
 \Eop_x^{\pi^*}\left[ \left(y+\sum_{k=0}^{(n-1) \wedge (\tau-1)}\beta^k f^*(x_k,y_k)\right)^\gamma
1[\tau \le n]\right].
\end{eqnarray}
From Lemma \ref{lem:pbounds} and the concavity of $x\mapsto x^\gamma,$
we obtain the following bound for the first term in (\ref{eq1})
\begin{eqnarray}
\label{eq2} \nonumber
\lefteqn{
(\beta^\gamma)^n\Eop_x^{\pi^*}\left(J(x_n,y_n)1[\tau>n]\right)\le(\beta^\gamma)^n
\Eop_x^{\pi^*}\left(\left(x_n+y_n+\frac\beta{1-\beta}\Eop Z^+\right)^\gamma1[\tau>n]\right)}\\
&\le& (\beta^\gamma)^n\Eop_x^{\pi^*}\left(x_n^\gamma1[\tau>n]\right) +(\beta^\gamma)^n
\left(\frac\beta{1-\beta}\Eop Z^+\right)^\gamma
+(\beta^\gamma)^n\Eop_x^{\pi^*}\left((y_n)^\gamma 1[\tau>n]\right)
\end{eqnarray}
By  (\ref{s_pol_y}) the third term in (\ref{eq2}) can be written as follows
 \begin{equation}
\label{eq3}
(\beta^\gamma)^n\Eop_x^{\pi^*}\left((y_n)^\gamma 1[\tau>n]\right)=\Eop_x^{\pi^*}
\left[ \left(y+\sum_{k=0}^{n-1}\beta^k f^*(x_k,y_k)\right)^\gamma
1[\tau> n]\right].
\end{equation}
Next by (\ref{maineq}) we get for  the first term in (\ref{eq2})  the following
\begin{eqnarray}
\label{eq4}
0&\le& (\beta^\gamma)^n\Eop_x^{\pi^*}\left(x_n^\gamma1[\tau>n]\right)
\le\left(\beta^n\Eop_x^{\pi^*}(x_n1[\tau>n])\right)^\gamma \\ \nonumber
&\le&  \left(\beta^n\Eop_x^{\pi^*}\left[\left(x-\sum_{k=0}^{n-1}f^*(x_k,y_k)
+\sum_{k=1}^{n}Z_k^+\right)1[\tau>n]\right]\right)^\gamma \\ \nonumber
&\le&  \left(\beta^n x +\beta^n n\Eop Z^+\right)^\gamma.
\end{eqnarray}
 By our assumption (A1) and (\ref{eq4}) the first term in (\ref{eq2}) converges to $0$ as $n\to\infty.$
 Observe that the same
 remark also applies for the second term in (\ref{eq2}). Summing up, from (\ref{eq1}), (\ref{eq2}) and
 (\ref{eq3}) we obtain that
 $$J(x,y) \le \lim_{n\to\infty}  \Eop_x^{\pi^*}\left(y+\sum_{k=0}^{(n-1) \wedge (\tau-1)}\beta^k f^*(x_k,y_k)\right)^\gamma.$$
 Now the monotone convergence theorem yields that
$J(x,y)\le J_{\pi^*}(x,y).$
\end{proof}

We close this section with a conclusion for our original model.

\begin{corollary}\label{optimal_policy_1}
Let $y_0:=0.$ Then, $\pi^*$ is optimal for the original
  optimisation problem, i.e.,
  $\widehat{J}(x) = \widehat{J}_{\pi^*}(x).$
\end{corollary}

\begin{remark}
Note that in the case of a power utility, the optimal policy is history-dependent,
 but depends on the history only through the accumulated discounted dividends given by $(y_k)$ in \eqref{s_pol_y}.
\end{remark}

\begin{remark}
It is well-known that the logarithmic utility function $U(x)=\log(x)$ can be obtained as a limit from the power utility since
$$ \lim_{\gamma\to 0} \frac1\gamma \big(x^\gamma-1\big)=\log(x).$$
Indeed, the problem can then be treated for the logarithmic utility in a similar way. The optimality equation is given by
$$ J(x,y) = \log(\beta) + \max_{a\in A(x)}
\left[ \sum_{k=-\infty}^\infty J\left(x-a+k, \frac{a+y}\beta\right) q_k \right]$$
and we can follow the same line of analysis. It is worth mentioning that the power and logarithmic utility functions
are examples of the so-called HARA ({\it hyperbolic absolute risk aversion}) utilities, whereas the exponential utility function belongs to
the CARA ({\it  constant absolute risk aversion}) class of utilities. The reader is referred, for instance,  to \cite{br,br_mor, turkey,fs}
and references cited therein, for
further properties of the aforementioned functions.
\end{remark}

\section{Howard's Policy Improvement}
In this section we provide one numerical tool to solve these problems which is known under the name Howard policy improvement. We restrict the presentation here to the exponential utility function: Start with an arbitrary policy of the form $\pi = (f(\cdot,\gamma), f(\cdot,\gamma\beta),\ldots)$ induced by a decision rule $f$ where we assume that $f$ is such that $f(x,\theta) \ge x-s^*$ for all $x>s^*$ and all $\theta$. Note that we define $s^* := \sup_{\theta \in [\gamma,0)}s(\theta)$ (see Theorem \ref{theo:xi_finite}). Take e.g. $f(x,\theta)=x^+$. We write $J_f:= J_\pi$. Next we compute the largest minimizer $h(x,\theta)$ of the expression
$$ a \mapsto e^{\theta a} G_f(x-a,\theta \beta),\; a\in\{0,1,\ldots ,x\}$$
where $$G_f(x,\theta) := \sum_{k=-x}^\infty J_f(x+k,\theta) q_k + \sum_{k=-\infty}^{-x-1} q_k.$$
We claim now that
\begin{lemma}
The new decision rule $h$ and the corresponding value function $J_h$ have the following properties:
\begin{itemize}
\item[a)] $h\big(x-h(x,\theta),\theta\big)=0$ for all $x$ and $\theta$.
  \item[b)] $h(x,\theta) \ge x-s^*$ for all $x>s^*$ and all $\theta$.
  \item[c)] $e^{\theta x} \underline{h}(\theta)\le J_h \le J_f\le e^{\theta x}  \overline{h}(\theta).$
\end{itemize}
\end{lemma}

\begin{proof}
\begin{itemize}
\item[a)] If $h(x,\theta)=0$ or $h(x,\theta)=x$ the statement is true. Now let $0<h(x,\theta)<x$ and suppose that $h\big(x-h(x,\theta),\theta\big)>0$, i.e. there is an $a^*>0$ s.t.
    $$ e^{\theta a^*} G_f\big(x-h(x,\theta)-a^*,\theta \beta\big) \le G_f\big(x-h(x,\theta),\theta \beta\big).$$
    On the other hand by the definition of $h$ we have for all $a>h(x,\theta)$:
   $$ e^{\theta h(x,\theta)} G_f\big(x-h(x,\theta),\theta \beta\big) < e^{\theta a} G_f\big(x-a,\theta \beta\big).$$
  Combining these inequalities leads to (note that $x-h(x,\theta)-a^*\ge 0$)
  \begin{eqnarray*}
     e^{\theta h(x,\theta)} G_f\big(x-h(x,\theta),\theta \beta\big) & <& e^{\theta (h(x,\theta)+a^*)} G_f\big(x-h(x,\theta)-a^*,\theta \beta\big) \\
     &\le &  e^{\theta h(x,\theta)} G_f\big(x-h(x,\theta),\theta \beta\big)
  \end{eqnarray*}
  which is a contradiction. Thus, the statement is shown.
  \item[b)] We show first for $x>s^*$ and arbitrary $\theta$ that $h(x,\theta)>0$. In order to do this, consider the expression
  $ e^{\theta a} G_f(x-a,\theta \beta)$ for $a=0$ and $a=f(x,\theta)$. By definition we obtain for $a=f(x,\theta)$ that
      $$  e^{\theta f(x,\theta)} G_f\big(x- f(x,\theta),\theta \beta\big) =J_f(x,\theta) \le e^{\theta x} \bar{h}(\theta).$$
      For $a=0$ we obtain that
      \begin{eqnarray*}
              G_f(x,\theta \beta) &=& \sum_{k=-x}^\infty J_f(x+k,\theta\beta) q_k + \sum_{k=-\infty}^{-x-1} q_k \\
         &\ge & e^{\theta \beta x} \underline{h}(\theta\beta)  \Big( \sum_{k=-x}^\infty e^{\theta\beta k} q_k + \sum_{k=-\infty}^{-x-1} q_k\Big)\\
         &=& e^{\theta \beta x} \underline{h}(\theta\beta) \Eop e^{\theta \beta Z^+} = e^{\theta \beta x} \underline{h}(\theta).
      \end{eqnarray*}
  Furthermore, observe that
  $$ e^{\theta \beta x} \underline{h}(\theta) \ge e^{\theta x} \bar{h}(\theta) \quad \Leftrightarrow \quad x>s(\theta).$$
  Thus, the inequality holds, in particular, if $x \ge s^* = \sup_{\theta \in [\gamma,0)}s(\theta)$. This implies that $0$ cannot be a minimiser,
  so $h(x,\theta)>0$ for all  $x>s^*$ and all $\theta$.
   This fact and point (a) imply the conclusion.
  \item[c)] From the definition of $h$ we obtain:
    $$ J_f(x,\theta) =  e^{\theta f(x,\theta)} G_f\big(x- f(x,\theta),\theta \beta\big) \ge  e^{\theta h(x,\theta)} G_f\big(x- h(x,\theta),\theta \beta\big).$$ Iterating this inequality yields
    $$ J_f(x,\theta) \ge \Eop_x^h \Big[ \exp\Big(\gamma \sum_{k=0}^{(\tau-1)\wedge (n-1)} \beta^k a_k\Big) (J_f(x_n,\gamma \beta^n)
1[\tau \ge n] +1[\tau < n])\Big].$$ The property of $h$ shown in b) now implies that ruin occurs with probability 1 under $h$ and thus as in the proof of Theorem \ref{theo:optimal} we obtain with $n\to\infty$ that $J_f \ge J_h$.

\end{itemize}
\end{proof}

From the proof it follows that in case $f\neq h$, the inequality $J_h(x,\theta) \le J_f(x,\theta)$ is strict for at least one $(x,\theta)$.
Now suppose no improvement is possible, i.e. $h=f$. Hence $J_f$ is another solution of \eqref{eq:fixedpoint}. By
Remark 3.2 $J_f \le J$. On the other hand by the definition of $J$ we have $J_f \ge J$ which implies $J=J_f$.

Finally if the iteration does not stop we obtain a non-increasing sequence $J_{f_0} \ge J_{f_1}\ge \ldots \ge J$.
Denote $\underline{J} := \lim_{n\to\infty} J_{f_n}$. Obviously $\underline{J} \ge J$.
Next  from the definition of an improvement:
$$ J_{f_{k+1}}(x,\theta) \le \min_{a \in A(x)}\Big[ e^{\theta a} G_{f_k}(x-a,\theta \beta)\Big]   \le J_{f_k}(x,\theta).$$
Letting $k\to\infty$ we obtain (note that $\lim$ and $\min$ can be interchanged since $A(x)$ is always finite):
$$\underline{J}(x,\theta) \le  \min_{a\in A(x)}\left[e^{\theta a}
\left(\sum_{k=a-x}^\infty \underline{J}(x-a+k,\theta\beta)q_k+ \sum_{k=-\infty}^{a-x-1} q_k\right)\right]
  \le \underline{J}(x,\theta)$$ hence $\underline{J}$ is another solution of \eqref{eq:fixedpoint} which implies that $\underline{J}=J$.

\section {Concluding remarks}

In this paper, we study the discrete time problem, suggested by Gerber and Shiu \cite{gershiu}, of maximising
the {\it expected utility} of discounted dividends until ruin.  We restrict our attention to the integer-valued surplus process
and to integer payments.  To the best of our knowledge, the only paper that examines a similar issue (with the exponential utility)
is \cite{schach}, where the wealth of insurance company is driven by a Brownian motion with drift. However, the authors
have not been able to solve the problem rigorously. Namely, assuming that a certain integral equation for the barrier function $b(t)$
has a desirable solution
(see Standing Assumption in \cite{schach}), they prove that $b(t)$ is indeed the barrier they search for
(a barrier function is a band function with $n=0$
in Definition 3.1).
Moreover, the numerical experiments provided in Section 1.4 in \cite{schach} are given without their convergence proofs.
This fact and the lack of a complete solution in continuous time and any solution in discrete time to Gerber and Shiu's suggestion
since 2004 indicate that  the problem is not straightforward from the mathematical point of view.
Firstly, similar as in \cite{br_mor,schach} we note that the optimal strategy is time dependent in a certain way. In order to get rid of
non-stationarity we extend the state space to the two-dimensional space.  Within such a new framework our problem becomes stationary.
Secondly, since our dividend payments can be unbounded we cannot directly apply the results from \cite{br_mor, ds1} to deduce that the
value function satisfies the corresponding Bellman equation for the exponential and power utility functions. Nonetheless,
we are able to show that in both cases the value iteration algorithm works (see Theorem 3.1 and Theorem 4.1) and in the exponential
function case the Howard's policy improvement algorithm works (see Section 5). These facts, in turn,
may have a significant meaning, when  one thinks of
numerical examples. Moreover, we are also able to describe the structure of optimal strategies for both utility cases
and to prove for the exponential function case that the optimal policy is a  band-policy.

Numerical experiments are difficult. Let us first recall that
the maximisation of the expectation of  discounted dividends in the
model given by  (\ref{maineq}) with $\Pop(Z_1=1)=p=1-\Pop(Z_1=-N),$ where $p\in (0,1)$ and $N\in \N,$
was a challenging  analytical problem. The reader may find the complete solution, for instance, in \cite{morill},
where it was shown that the optimal
policy is of barrier type.
Our problem, as already mentioned,  is non-stationary and non-separable.
Therefore, the methods that solved analytically  the risk neutral problem are useless here. Moreover, as was noted by
  Gerber and Shiu \cite{gershiu},
in contrast to the risk neutral problem one can expect that in the model with exponential function the optimal policy is not of  barrier type.
This fact does not make easier potential calculations.
Obviously, since obtaining an analytical solution is a  challenge, one can think of numerical methods used in dynamic programming such as
value iteration, policy improvement or others, see \cite{powell}.
However, our surplus process proceeds on the space $X=\Z$ and even the simple aforementioned case  ($\Pop(Z_1=1)=p=1-\Pop(Z_1=-1)$)
requires some truncation of the state space to a finite one in order to obtain numerical results.
In addition, in this model we meet one more obstacle that
have not been treated so far, namely the exponential and power utility functions that imply non-stationarity and non-separability.
Therefore,  the problem of calculating numerically optimal strategies and value functions
 for models with these or other utility functions is left open.

\small
\bibliographystyle{abbrv}

\begin{thebibliography}{99}

\bibitem{abt} H. Albrecher, N. B\"auerle, S. Thonhauser (2011). Optimal dividend-payout in random discrete time,
{\it Statistics and Risk Modeling} {\bf 28}, 251-276.

\bibitem{albrthonh} H. Albrecher, S. Thonhauser (2009). Optimality results for dividend problems in insurance,
{\it  RACSAM, Rev. R. Acad. Cienc. Exactas F\'is. Nat., Ser. A Mat.} {\bf 103}, 295-320.

\bibitem{at} S. Asmussen, M. Taksar (1997). Controlled diffusion models for optimal dividend payout,
{\it Insurance Math. Econom.} {\bf 20}, 1-15.

\bibitem{avanzi} B. Avanzi (2009). Strategies for dividend distribution: A review. {\it N. Am.  Actuar. J.}  {\bf 13}, 217-251.

\bibitem{br} N. B\"auerle, U. Rieder,  Markov Decision Processes with Application to Finance,
Springer, Berlin 2011.

\bibitem{br_mor} N. B\"auerle, U. Rieder (2014).
More risk-sensitive Markov decision processes, {\it Math. Oper. Res.} {\bf 39}, 105-120.

\bibitem{bpli} T. Bielecki, S. Pliska (2003). Economic properties of the risk sensitive criterion for portfolio management,
{\it Rev. Account. Fin.} {\bf 2}, 3-17.

\bibitem{cs} K.-J. Chung, M.J.  Sobel (1987). Discounted MDP's: distribution functions and exponential utility maximization,
{\it SIAM J. Control Optim.} {\bf 25}, 49-62.

\bibitem{turkey} E. \c{C}anako\u{g}lu, S. \"Ozekici (2010). Portfolio selection in stochastic markets
with HARA utility functions, {\it Europ. J. Oper. Res.} {\bf 201}, 520-536.

\bibitem{deFinetti} B. de Finetti (1957). Su un' impostazione alternativa della teoria collettiva del rischio, {\it Transactions
of the XVth Congress of Actuaries} {\bf (II)}, 433-443.

\bibitem{ds1} G.B. Di Masi, {\L}. Stettner (2000). Risk-sensitive control of discrete time Markov processes
with infinite horizon, {\it SIAM J. Control Optim.} {\bf 38}, 61-78.

\bibitem{fs} H. F\"ollmer, A. Schied. Stochastic Finance, An Introduction in Discrete Time. Walter de
Gruyter, Berlin 2011.

\bibitem{gershiu} H.U. Gerber, E.S.W. Shiu (2004). Optimal dividends: analysis with Brownian motion, {\it N.
Am. Actuar. J.} {\bf  8}, 1-20.

\bibitem{schach}  P. Grandits, F. Hubalek, W. Schachermayer, M. \v{Z}igo (2007).
Optimal expected exponential utility of dividend payments in Brownian risk model,
{\it Scand. Actuar. J.} {\bf 2}, 73-107.

\bibitem{j1} S.C. Jaquette (1973). Markov decision processes with a new optimality criterion: discrete time,
{\it Ann. Statist.} {\bf 1}, 496-505.

\bibitem{j2} S.C. Jaquette (1976). A utility criterion for Markov decision processes, {\it Management
Sci.} {\bf 23},  43-49.

\bibitem{kkm} Y. Kadota, M. Kurano, M. Yasuda (1998).  On the general utility of discounted Markov decision processes,
{\it Int. Trans. Oper. Res.} {\bf 15}, 27-34.

\bibitem{morill} J.E. Morrill (1966). One-person game of economic survival, {\it Nav. Res. Logistics Quarterly} {\bf 13}, 49-69.

\bibitem{n} J. Neveu, Mathematical Foundations of the Calculus of Probability, Holden-Day, San Francisco,
1965.

\bibitem{powell}  W.B. Powell, Approximate Dynamic Programming, Solving the Curses of Dimensionality, John Wiley, New Jersey, 2011.

\bibitem{s} H. Schmidli,  Stochastic Control in Insurance, Springer, London, 2008.

\bibitem{ejor} D. Yao, H. Yang, R. Wang (2011). Optimal dividend and capital injection problem in the dual model
with proportional and fixed transaction costs, {\it Europ. J. Oper. Res.} {\bf 211}, 568-576.


\end{thebibliography}

\end{document}